\documentclass{amsart}

\newtheorem{theorem}{Theorem}
\newtheorem{lemma}[theorem]{Lemma}

\newtheorem{corollary}[theorem]{Corollary}

\theoremstyle{definition}
\newtheorem{definition}[theorem]{Definition}

\usepackage{amscd,amssymb}

\begin{document}

\title[Weitzenb\"ock derivations of polynomial algebras]
{The Conjecture of Nowicki on Weitzenb\"ock derivations of polynomial algebras}

\author[Vesselin Drensky and Leonid Makar-Limanov]
{Vesselin Drensky and Leonid Makar-Limanov}
\address{Institute of Mathematics and Informatics,
Bulgarian Academy of Sciences,
1113 Sofia, Bulgaria}
\email{drensky@math.bas.bg}
\address{Department of Mathematics,
Wayne State University
Detroit, MI 48202, USA.}
\email{lml@math.wayne.edu}

\thanks
{The research of the first author was partially supported by Grant
MI-1503/2005 of the Bulgarian National Science Fund.}

\thanks
{The work of the second author was partially supported by an NSA grant.}


\begin{abstract} The Weitzenb\"ock theorem states that if $\Delta$
is a linear locally nilpotent derivation of
the polynomial algebra $K[Z]=K[z_1,\ldots,z_m]$ over a field $K$
of characteristic 0, then the algebra of constants of $\Delta$ is finitely generated.
If $m=2n$ and the Jordan normal form of $\Delta$ consists of $2\times 2$ Jordan cells only,
we may assume that $K[Z]=K[X,Y]$ and $\Delta(y_i)=x_i$, $\Delta(x_i)=0$, $i=1,\ldots,n$.
Nowicki conjectured that the algebra of constants $K[X,Y]^{\Delta}$
is generated by $x_1,\ldots,x_n$ and $x_iy_j-x_jy_i$, $1\leq i<j\leq n$.
Recently this conjecture was confirmed in the Ph.D. thesis of Khoury, and also by Derksen.
In this paper we give an elementary proof of the conjecture of Nowicki. Then we find a very simple system
of defining relations of the algebra $K[X,Y]^{\Delta}$ which corresponds to the reduced Gr\"obner basis
of the related ideal with respect to a suitable admissible order, and present an explicit basis of
$K[X,Y]^{\Delta}$ as a vector space.
\end{abstract}

\maketitle

\section*{Introduction}

Let $K$ be a field of characteristic 0 and let $K[Z]=K[z_1,\ldots,z_m]$ be the polynomial $K$-algebra
in $m$ variables. A linear operator $\Delta$ of $K[Z]$ is called a derivation if
$\Delta(uv)=\Delta(u)v+u\Delta(v)$ for all $u,v\in K[Z]$. The derivation $\Delta$ is locally nilpotent
if for each $u\in K[Z]$ there exists a $d\geq 1$ such that $\Delta^d(u)=0$.
Locally nilpotent derivations of polynomial algebras are subjects of active investigation.
They play essential role in the study of automorphism group of $K[Z]$,
including the generation of $\text{Aut }K[x,y]$ by
tame automorphisms, the Jacobian conjecture, in invariant theory, Fourteenth Hilbert's problem
and other important topics. See the books by Nowicki \cite{N}, van den Essen \cite{E},
and Freudenburg \cite{F} for details.

The well known theorem of Weitzenb\"ock \cite{W} states that
if $\Delta$ is a nilpotent linear operator acting on the $m$-dimensional vector space
$KZ=Kz_1\oplus \cdots\oplus Kz_m$ and we extend it to a derivation of $K[Z]$,
then the algebra $K[Z]^{\Delta}$ of constants of $\Delta$, i.e., the kernel of $\Delta$ in $K[Z]$,
is a finitely generated algebra. We call $\Delta$, which is a locally nilpotent derivation, a Weitzenb\"ock derivation.
A modern proof of the theorem of Weitzenb\"ock
is given by Seshadri \cite{S}, with further simplification by Tyc \cite{T}, see also \cite{N, F}.

Up to a change of the basis of the vector space $KZ$, Weitzenb\"ock derivations
are determined by their Jordan normal form. Each Jordan cell is an upper triangular matrix with zero diagonal.
Hence, for each dimension $m$ it is sufficient to consider a finite number of Weitzenb\"ock derivations $\Delta$.
There are algorithms which find a set of generators of $K[Z]^{\Delta}$
for a given $\Delta$. Nevertheless from computational point
of view it is difficult to find explicitly such a system. In particular, no upper bound
for the degree of the generators is known.

The algebra of constants of $\Delta$ coincides with the algebra of invariants of
the linear operator
\[
\exp(\Delta)=1+\frac{\Delta}{1!}+\frac{\Delta^2}{2!}+\cdots
\]
and $K[Z]^{\Delta}$ may be studied also with methods of invariant theory.
In particular, when we find the generators of the algebra of invariants
of $\exp(\Delta)$, this is stated as the First fundamental theorem
of the invariants of $\exp(\Delta)$. When we find the defining relations,
this is the Second fundamental theorem.

When the Jordan normal form of the Weitzenb\"ock derivation $\Delta$ consists of $2\times 2$ Jordan cells only,
we may assume that
\[
K[Z]=K[X,Y]=K[x_1,\ldots,x_n,y_1,\ldots,y_n],
\]
\[
\Delta=\sum_{i=1}^nx_i\frac{\partial}{\partial y_i}
\]
and hence the action of $\Delta$ is defined by
\[
\Delta(x_i)=0,\quad \Delta(y_i)=x_i,\quad i=1,\ldots,n.
\]
Till the end of the paper we shall consider such derivations only.
Nowicki \cite{N} conjectured that $K[X,Y]^{\Delta}$ is generated by
$x_1,\ldots,x_n$ and the determinants
\[
u_{ij}=\left\vert\begin{matrix}
x_i&x_j\\
y_i&y_j\\
\end{matrix}\right\vert,\quad 1\leq i<j\leq n.
\]
The conjecture of Nowicki was confirmed by Khoury in his Ph.D. thesis \cite{K}. His proof
is very computational and uses essentially Gr\"obner basis techniques. Another proof was given
(but not published yet) by Derksen who combined ideas of the proof of Seshadri \cite{S} of
the Weitzenb\"ock theorem with the explicit form of the invariants of the special linear
group $SL_2(K)$ acting on a direct sum of two-dimensional $SL_2(K)$-invariant vector spaces.

The purpose of our paper is to give a new elementary proof of the conjecture of Nowicki.
We use easy arguments from undergraduate algebra and a simple induction only.
We find also a uniformly looking explicit set of defining relations
of the algebra of constants $K[X,Y]^{\Delta}$ which corresponds to the reduced Gr\"obner basis
of the related ideal of $K[X,U]$, where $U=\{u_{ij}\mid 1\leq i<j\leq n\}$, as well as a basis
of $K[X,Y]^{\Delta}$ as a vector space.

\section{The conjecture of Nowicki}

We fix the sets of variables
\[
X'=\{x_1,\ldots,x_{n-1}\},\quad Y'=\{y_1,\ldots,y_{n-1}\}
\]
\[
X=X'\cup\{x_n\},\quad Y=Y'\cup\{y_n\}
\]
and the derivation of $K[X,Y]=K[X',Y',x_n,y_n]$
\[
\Delta=\sum_{i=1}^nx_i\frac{\partial}{\partial y_i}.
\]
The derivation $\Delta$ acts in the same manner on all $x_i$ and $y_i$.
Hence the $K$-algebra endomorphism $\varphi_{\alpha}$ of $K[X,Y]$, $\alpha=(\alpha_1,\ldots,\alpha_{n-1})\in K^{n-1}$,
defined by
\[
\varphi_{\alpha}(x_i)=x_i,\quad \varphi_{\alpha}(y_i)=y_i,\quad i=1,\ldots,n-1,
\]
\[
\varphi_{\alpha}(x_n)=\alpha_1x_1+\cdots+\alpha_{n-1}x_{n-1},\quad
\varphi_{\alpha}(y_n)=\alpha_1y_1+\cdots+\alpha_{n-1}y_{n-1},
\]
commutes with $\Delta$. If $f(X',Y',x_n,y_n)$ belongs to the algebra of constants $K[X,Y]^{\Delta}$, then
\[
\varphi_{\alpha}(f)=f(\varphi_{\alpha}(X'),\varphi_{\alpha}(Y'),
\varphi_{\alpha}(x_n),\varphi_{\alpha}(y_n))
\]
\[
=f(X',Y',\alpha_1x_1+\cdots+\alpha_{n-1}x_{n-1},\alpha_1y_1+\cdots+\alpha_{n-1}y_{n-1})
\in K[X',Y']^{\Delta}
\]
for all $\alpha=(\alpha_1,\ldots,\alpha_{n-1})\in K^{n-1}$.

\begin{lemma}\label{main lemma}
Let $n\geq 2$ and let a nonzero polynomial $f=f(X',Y',x_n,y_n)$ be homogeneous
with respect to $x_n,y_n$. If
\[
\varphi_{\alpha}(f)=
f(X',Y',\alpha_1x_1+\cdots+\alpha_{n-1}x_{n-1},\alpha_1y_1+\cdots+\alpha_{n-1}y_{n-1})=0
\]
for an $\alpha=(\alpha_1,\ldots,\alpha_{n-1})\in K^{n-1} \setminus 0$, then
$f$ is divisible by
\[
w_{\alpha}(X',Y',x_n,y_n)=(\alpha_1x_1+\cdots+\alpha_{n-1}x_{n-1})y_n-(\alpha_1y_1+\cdots+\alpha_{n-1}y_{n-1})x_n.
\]
\end{lemma}

\begin{proof}
Let
\[
0\not= f = a_py_n^p+a_{p-1}x_ny_n^{p-1}+\cdots+a_1x_n^{p-1}y_n+a_0x_n^p,
\quad a_i\in K[X',Y'],
\]
be a polynomial homogeneous in $x_n$, $y_n$. Then
\[
(\alpha_1x_1+\cdots+\alpha_{n-1}x_{n-1})f = a_p w_{\alpha}y_n^{p-1}
\]
\[
+ a_p (\alpha_1y_1+\cdots+\alpha_{n-1}y_{n-1})x_ny_n^{p-1} + (\alpha_1x_1+\cdots+\alpha_{n-1}x_{n-1})x_n g,
\]
where $g = a_{p-1}y_n^{p-1} + a_{p-2}x_ny_n^{p-2}+\cdots+a_1x_n^{p-2}y_n+a_0x_n^{p-1}$.
Since $\varphi_{\alpha}(a_p (\alpha_1y_1+\cdots+\alpha_{n-1}y_{n-1})y_n^{p-1} + g) = 0$
we can assume by induction on the degree of a polynomial relative to $x_n, \, y_n$
that $w_{\alpha}$ divides $a_p (\alpha_1y_1+\cdots+\alpha_{n-1}y_{n-1})y_n^{p-1} + g$
(the base of induction for the degree equal to zero is obviously correct) and
so $w_{\alpha}$ divides $(\alpha_1x_1+\cdots+\alpha_{n-1}x_{n-1})f$.

The polynomial $w_{\alpha}(X',Y',x_n,y_n)$
is irreducible in $K[X',Y',x_n,y_n]$.
Since
$\alpha_1x_1+\cdots+\alpha_{n-1}x_{n-1}$ does not depend on $x_n,y_n$, we conclude that
$w_{\alpha}(X',Y',x_n,y_n)$ divides $f(X',Y',x_n,y_n)$ in $K[X',Y',x_n,y_n]$.
\end{proof}

\begin{corollary}\label{corollary of Lemma 1}
If $\varphi_{\alpha}(f)= 0$
for all $\alpha=(\alpha_1,\ldots,\alpha_{n-1})\in K^{n-1} \setminus 0$, then
$f = 0$ if $n > 2$ and $f$ is divisible by $u_{12} = x_1 y_2 - x_2 y_1$ if $n = 2$.
\end{corollary}

\noindent {\it Explanation.} If $n>2$ then $f$ is divisible by $(x_1 + \alpha_2x_2)y_n - (y_1 + \alpha_2y_2)x_n$
where $\alpha_2$ is any element of $K$. So $f$ is divisible by infinitely many
pairwise non-proportional irreducible polynomials. This is of course impossible if $f \neq 0$.
If $n = 2$ then by Lemma \ref{main lemma},
$f$ is divisible by $u_{12}$ and in this case all $w_{\alpha}$
with non-zero $\alpha$ are proportional to $u_{12}$. \hfill $\Box$

\begin{theorem}\label{conjecture of Nowicki}
Let $X=\{x_1,\ldots,x_n\}$, $Y=\{y_1,\ldots,y_n\}$,
$n\geq 1$, and let $\Delta$ be the Weitzenb\"ock derivation of $K[X,Y]$ defined by
\[
\Delta=\sum_{i=1}^nx_i\frac{\partial}{\partial y_i}.
\]
The algebra of constants $K[X,Y]^{\Delta}$ is generated by
\[
x_i,\ i=1,\ldots,n,\quad
u_{ij}=x_iy_j-x_jy_i,\ 1\leq i<j\leq n.
\]
\end{theorem}

\begin{proof}
Let
\[
f=f(X',Y',x_n,y_n)\in K[X,Y]^{\Delta}.
\]
For a monomial $v\in K[X,Y]$ which is of degree $(d_1,d_2)$ with respect to $(X,Y)$,
i.e., of degree $d_1$ in $X$ and $d_2$ in $Y$, the image $\Delta(v)$ is of degree
$(d_1+1,d_2-1)$. Hence, if $f\in K[X,Y]^{\Delta}$, then the homogeneous in $(X,Y)$
components of $f$ are also in $K[X,Y]^{\Delta}$ and we may assume that $f$ is homogeneous
in $(X,Y)$. Similarly, if a monomial $u$ is of total degree $p$ in $x_n,y_n$, the same is true for $\Delta(u)$.
Again, we may assume that $f$ is homogeneous in $x_n,y_n$. We shall prove the theorem by induction
on $n$ and on the total degree in $x_n,y_n$. If $n=1$, then $K[x_1,y_1]^{\Delta}=K[x_1]$. We
assume that $n>1$ and $K[X',Y']^{\Delta}$ is generated by $X'$ and $U'=\{u_{ij}\mid 1\leq i<j\leq n-1\}$.
Since $\text{deg}_X(u_{ij})=\text{deg}_Y(u_{ij})=1$ and
$1=\text{deg}_X(x_i)>\text{deg}_Y(x_i)=0$, the homogeneous in $(X,Y)$ elements
$v\in K[X',Y']^{\Delta}$ satisfy $\text{deg}_{X'}(v)\geq \text{deg}_{Y'}(v)$.
If $\text{deg}_{x_n,y_n}(f)=0$, then $f=f(X',Y')\in K[X',Y']^{\Delta}$. Hence we may consider the
case $\text{deg}_{x_n,y_n}(f)>0$. We write $f$ in the form
\[
f=(a_py_n^p+a_{p-1}x_ny_n^{p-1}+\cdots+a_1x_n^{p-1}y_n+a_0x_n^p)x_n^q,
\]
where $a_j=a_j(X',Y')\in K[X',Y']$ and $a_p\not=0$. Since both $f$ and $x_n^q$ belong to
$K[X,Y]^{\Delta}$, the same holds for $f/x_n^q$ and we may assume that $q=0$.
The next observation is that $a_p\in K[X',Y']^{\Delta}$ because
\[
0=\Delta(f)=\Delta(a_p)y_n^p+(pa_p+\Delta(a_{p-1}))x_ny_n^{p-1}+\cdots+(a_1+\Delta(a_0))x_n^p.
\]
Hence $a_p$ has the form
\[
a_p=\sum x_{s_1}\cdots x_{s_k}b_s(U').
\]
If $\varphi_{\alpha}(f)=0$ for all $\alpha\in K^{n-1}$, then
Corollary \ref{corollary of Lemma 1} gives that
$n=2$ and $f$ is divisible by $u_{12}$. Hence $f_1=f/u_{12}$ also belongs to $K[X,Y]^{\Delta}$.
Since $\text{deg}_{x_2,y_2}(f_1)=\text{deg}_{x_2,y_2}(f)-1$, we apply inductive arguments
and conclude that $f\in K[x_1,x_2,u_{12}]$. Now we consider the case when
$\varphi_{\alpha}(f)\not=0$ for some $\alpha\in K^{n-1}$. The operators $\varphi_{\alpha}$
and $\Delta$ commute, so we have that $\varphi_{\alpha}(f)\in K[X',Y']^{\Delta}$. Also clearly
\[
\text{deg}_X(f)=\text{deg}_{X'}(\varphi_{\alpha}(f))
\geq \text{deg}_{Y'}(\varphi_{\alpha}(f))=\text{deg}_Y(f).
\]
Hence
\[
\text{deg}_{X}(f)=\text{deg}_{X'}(a_p)\geq \text{deg}_Y(f)=p+\text{deg}_{Y'}(a_p)\, \,\, \text{and}
\]
\[
f=\sum x_{s_1}\cdots x_{s_p}c_s(X',U')y_n^p+\sum_{i=0}^{p-1}a_i(X',Y')x_n^{p-i}y_n^i.
\]
We rewrite $x_{s_j}y_n$ in the form
\[
x_{s_j}y_n=(x_{s_j}y_n-x_ny_{s_j})+x_ny_{s_j}
=u_{s_jn}+x_ny_{s_j}
\]
and obtain
\[
x_{s_1}\cdots x_{s_p}y_n^p=\prod_{j=1}^p(u_{s_jn}+x_ny_{s_j})
=\prod_{j=1}^pu_{s_jn}+x_ng(X',Y',x_n,y_n)
\]
for some $g(X',Y',x_n,y_n)\in K[X',Y',x_n,y_n]$. Hence
\[
f=\sum c_s(X',U')\left(\prod_{j=1}^pu_{s_jn}+x_ng(X',Y',x_n,y_n)\right)
+\sum_{i=0}^{p-1}a_i(X',Y')x_n^{p-i}y_n^i
\]
\[
=\sum c_s(X',U')\prod_{j=1}^pu_{s_jn}+x_nf_1(X',Y',x_n,y_n)
\]
for some $f_1(X',Y',x_n,y_n)\in K[X',Y',x_n,y_n]$.
Since $f,x_n,c_s(X',U')\prod_{j=1}^pu_{s_jn}\in K[X,Y]^{\Delta}$, the same is true for
$f_1(X',Y',x_n,y_n)$ and we apply induction on the degree of $f_1$ in $x_n,y_n$. This completes the proof
of the theorem.
\end{proof}

\section{Defining relations}

In this section we shall show that the defining relations
of the algebra of constants $K[X,Y]^{\Delta}$, with respect to the
generators $x_i$, $u_{ij}=x_iy_j-x_jy_i$ from the conjecture
of Nowicki, consists of the relations
\[
r(i,j,k,l)= u_{ij}u_{kl} - u_{ik}u_{jl}+ u_{il}u_{jk} = 0,\quad
1\leq i<j<k<l\leq n,
\]
\[
s(i,j,k)= x_iu_{jk} - x_ju_{ik} + x_ku_{ij} = 0,\quad
1\leq i<j<k\leq n.
\]
Our proof will give that the elements $r(i,j,k,l)$ and $s(i,j,k)$
form the reduced Gr\"obner basis of the corresponding ideal
of $K[X,U]$, $U=\{u_{ij}\mid 1\leq i<j\leq n\}$,
with respect to a suitable admissible order. This provides also
a basis of $K[X,Y]^{\Delta}$ as a vector space.

It will be convenient to identify the generator $u_{ij}\in U$
with the open interval $(i,j)$ on the real line and define the interval length
of $u_{ij}$ by $\vert u_{ij}\vert=j-i$. We say that $u_{ij}$ and
$u_{kl}$ intersect each other
if the intervals $(i,j)$ and $(k,l)$ have a nonempty
intersection and are not contained in each other. This means that
one of the inequalities $i<k<j<l$ and $k<i<l<j$ holds. We say also that
$u_{ij}$ covers $x_k$ if $k$ belongs to the open interval $(i,j)$.

We summarize the necessary background on Gr\"obner bases. For more
details we refer
for example to the book by Adams and Loustaunau \cite{AL}.

Let $Z=\{z_1,\ldots,z_m\}$ be a set of variables. The linear ordering
$\succ$ on the set of monomials $[Z]$ is admissible, if
it satisfies the descending chain condition
and $u\succ v$ for $u,v\in [Z]$ imlies $uw\succ vw$ for all $w\in [Z]$.
Every nonzero polynomial $f(Z)$ of $K[Z]$ can be written as
\[
f=\beta_1v_1+\beta_2v_2+\cdots+\beta_kv_k, \quad
0\not=\beta_j\in K,
\]
where $v_1\succ v_2\succ\cdots\succ v_k$.
We denote by $\overline{f}$ the leading monomial $v_1$ of $f$ in $[Z]$.

Let $J$ be an ideal of $K[Z]$
and let $I(J)$ be the set of leading monomials of $J$.
A generating set $G$ of $J$ is called a Gr\"obner basis of $J$
(with respect to the fixed admissible order on $[Z]$)
if for any $f\in J$ there exists an $f_i\in G$ such that
$\overline{f}$ is divisible by $\overline{f_i}$.
Equivalently, the set $I(G)$ of leading monomials of $G$
generates the semigroup ideal $I(J)$ of $[Z]$.
A subset $G$ of $J$ is a Gr\"obner basis of $J$ if and only if it
has the following property. The set
of normal monomials of $[Z]$ with respect to $G$, i.e., the monomials which are
not divisible by an element of $I(G)$, forms a $K$-basis of the factor
algebra $K[Z]/J$. The set of normal monomials with respect to $G$
spans $K[Z]/J$ for any subset $G$ of $J$. Hence $G\subset J$ is a Gr\"obner basis of $J$
if and only if the set of normal monomials with respect to $G$ is linearly
independent in the factor algebra $K[Z]/J$.
The Gr\"obner basis $G$ is reduced if
the monomials $v_i$ participating in each $f_j\in G$ are not divisible
by the leading monomials of $G\backslash \{f_j\}$.

It is easy to see that the following ordering is admissible.
We believe that it may be applied also to other problems.

\begin{definition}\label{deg-interval length-lex order}
We order the monomials of $K[X,U]$ first by the degree in $X$ and $U$,
then by the total interval length of the participating
variables $u_{ij}$ and finally lexicographically, as follows.

Let
\[
v=x_{i_1}\cdots x_{i_c}u_{j_1k_1}\cdots u_{j_dk_d},
\]
where $i_1\leq\cdots\leq i_c$, $j_1\leq\cdots\leq j_d$ and
$k_a\leq k_{a+1}$ if $j_a=j_{a+1}$, and
\[
v'=x_{i'_1}\cdots x_{i'_{c'}}u_{j'_1k'_1}\cdots u_{j'_{d'}k'_{d'}}
\]
with similar restrictions on $i'_a,j'_b,k'_b$.

We define $v\succ v'$
if

(i) $c>c'$ (the degree of $v$ in $X$ is bigger than the degree of $v'$ in $X$);

(ii) $c=c'$ and $d>d'$ (the degree of $v$ in $U$ is bigger than the degree of $v'$ in $U$);

(iii) $c=c'$, $d=d'$ and
\[
\sum_{b=1}^d\vert u_{j_bk_b}\vert> \sum_{b=1}^d\vert u_{j'_bk'_b}\vert
\]
(the total interval length of $v$ is bigger than that of $v'$);

(iv) $c=c'$, $d=d'$,
\[
\sum_{b=1}^d\vert u_{j_bk_b}\vert= \sum_{b=1}^d\vert u_{j'_bk'_b}\vert
\]
and $\omega\succ \omega'$
for the $(c+2d)$-tuples
\[
\omega=(i_1,\ldots, i_c,j_1,\ldots,j_d,k_1,\ldots,k_d),\quad
\omega'=(i'_1,\ldots, i'_c,j'_1,\ldots,j'_d,k'_1,\ldots,k'_d),
\]
where $(a_1,\ldots,a_p)\succ (b_1,\ldots,b_p)$
if $a_1=b_1,\ldots,a_e=b_e$, $a_{e+1}<b_{e+1}$ for some $e$.

We call this ordering degree--interval length--lexicographic order (DILL order) of $K[X,U]$.
\end{definition}

\begin{theorem}\label{defining relations}
Let $X=\{x_1,\ldots,x_n\}$, $Y=\{y_1,\ldots,y_n\}$,
$n\geq 1$, and let $\Delta$ be the Weitzenb\"ock derivation of $K[X,Y]$ defined by
\[
\Delta=\sum_{i=1}^nx_i\frac{\partial}{\partial y_i}.
\]

{\rm (i)} The algebra of constants has the presentation
\[
K[X,Y]^{\Delta}\cong
K[X,U]/(R,S),
\]
where $X=\{x_i\mid i=1,\ldots,n\}$,
$U=\{u_{ij}\mid 1\leq i<j\leq n\}$ and
the ideal $(R,S)$ is generated by
\[
R= \{r(i,j,k,l)= u_{ij}u_{kl} - u_{ik}u_{jl}+ u_{il}u_{jk}\mid
1\leq i<j<k<l\leq n\},
\]
\[
S = \{s(i,j,k)= x_iu_{jk} - x_ju_{ik} + x_ku_{ij}\mid
1\leq i<j<k\leq n\}.
\]
{\rm (ii)} The set $R\cup S$ is the reduced Gr\"obner basis of the ideal
$(R,S)$ with respect to the DILL order of $K[X,U]$.

{\rm (iii)} As a vector space $K[X,Y]^{\Delta}$ has a basis consisting
of all products
\[
x_{i_1}\cdots x_{i_c}u_{j_1k_1}\cdots u_{j_dk_d}
\]
such that the generators $u_{j_pk_p}$ and $u_{j_qk_q}$ do not intersect
each other and $u_{j_pk_p}$ does not cover $x_{i_t}$ for any $p,q,t$.
\end{theorem}

\begin{proof}
Let $\pi:K[X,U]\to K[X,Y]^{\Delta}$ be the homomorphism defined by
\[
\pi(x_i)=x_i,\quad \pi(u_{jk})=x_jy_k-x_ky_j,
\]
and let $J=\text{Ker}(\pi)\subset K[X,U]$.
We shall prove the following:

(1) The ideal $J$ contains $R$ and $S$;

(2) The set of normal monomials with respect to $R\cup S$ coincides with
the set of all products
\[
x_{i_1}\cdots x_{i_c}u_{j_1k_1}\cdots u_{j_dk_d}
\]
such that the generators $u_{j_pk_p}$ and $u_{j_qk_q}$ do not intersect
each other and $u_{j_pk_p}$ does not cover $x_{i_t}$ for any $p,q,t$;

(3) The images of these normal monomials under $\pi$ are linearly independent in $K[X,Y]$.

Since $K[X,Y]$ contains $K[X,Y]^{\Delta}=\pi(K[X,U])\cong K[X,U]/J$, the statement (3)
implies that the images of normal monomials under $\pi$ are linearly independent in $K[X,U]/J$
and so, as we mentioned above, $R\cup S$ is a Gr\"obner basis of $J=\text{Ker}(\pi)$.
We also check that this basis is reduced.

{\bf Step 1:}
It is easy to verify directly that $\pi(s(i, j, k)) = \pi(r(i, j, k, l)) = 0$.  Also
\[
\pi(s(i, j, k)) = \det \left(
          \begin{array}{ccc}
          x_i & x_j & x_k \\
          x_i & x_j & x_k \\
          y_i & y_j & y_k  \\
          \end{array}
          \right)
\]
expended relative to the first row and
\[
2\pi(r(i, j, k, l)) = \det \left(
\begin{array}{cccc}
 x_i & x_j & x_k & x_l \\
 y_i & y_j & y_k & y_l \\
 x_i & x_j & x_k & x_l \\
 y_i & y_j & y_k & y_l \\
 \end{array}
 \right)
\]
expended relative to the first two rows.

{\bf Step 2:} If $1\leq i<j<k<l\leq n$ then $\overline{r(i,j,k,l)}=u_{ik}u_{jl}$ since
$\vert u_{ik}u_{jl}\vert=\vert u_{il}u_{jk}\vert > \vert u_{ij}u_{kl}\vert$
and $(i,j,k,l)\succ (i,j,l,k)$. Similarly, if $1\leq i<j<k\leq n$ then
$\overline{s(i,j,k)}=x_ju_{ik}$ since $\vert u_{ik}\vert > \vert u_{jk}\vert$
and $\vert u_{ik}\vert > \vert u_{ij}\vert$. So different leading monomials do not divide
each other and the set $R \cup S$ is reduced.

Let
\[
v=x_{i_1}\cdots x_{i_c}u_{j_1k_1}\cdots u_{j_dk_d}
\]
be a normal monomial with respect to $R\cup S$.
If two variables $u_{ik}$ and $u_{jl}$ intersect each other, i.e., either
$1\leq i<j<k<l\leq n$ or $1\leq j<i<l<k\leq n$, then their product $u_{ik}u_{jl}$
is the leading monomial of $r(i,j,k,l)$ or $r(k,l,i,j)$ accordingly.
Hence in the normal monomial $v$ the participating generators $u_{j_pk_p}$
and $u_{j_qk_q}$ do not intersect each other.
Again, if
$u_{ik}$ covers $x_j$, i.e., $i<j<k$, then $x_ju_{ik}$ is the leading monomial of
$s(i,j,k)$. Hence $u_{j_pk_p}$ does not
cover $x_{i_t}$ in the normal monomial $v$ for any $p,q,t$.

{\bf Step 3:}
We want to show that the images of the normal monomials
under $\pi$ are linearly independent in $K[X,Y]$.

Introduce the following
ordering on the monomials of $K[X,Y]$:
\[
x_1^{a_1}y_1^{b_1}\cdots x_n^{a_n}y_n^{b_n}
>x_1^{a'_1}y_1^{b'_1}\cdots x_n^{a'_n}y_n^{b'_n},
\]
if $(a_1,b_1,\ldots,a_n,b_n)>(a'_1,b'_1,\ldots,a'_n,b'_n)$
lexicographically, i.e., either $a_1 > a'_1$ or $a_1=a'_1,b_1=b_1',\ldots,b_k=b'_k$,
and either $a_{k+1}>a'_{k+1}$
or $a_{k+1}=a'_{k+1}$, $b_{k+1}>b'_{k+1}$ for some $k$.
The leading monomial of $\pi(u_{jk})=x_jy_k-x_ky_j$, $j<k$, in $K[X,Y]$ is
$\text{lead}(u_{jk})= x_j y_k$ since $x_j^1y_j^0x_k^0y_k^1 > x_j^0y_j^1x_k^1y_k^0$
and the leading monomial of
$w=x_{i_1}\cdots x_{i_c}u_{j_1k_1}\cdots u_{j_dk_d}$ is
\[
\text{lead}(w)=x_{i_1}\cdots x_{i_c}x_{j_1}y_{k_1}\cdots x_{j_d}y_{k_d},
\]
so $\text{deg}_Y(\pi(w)) = \text{deg}_U(w)$.

We shall prove by induction on $\text{deg}_Y(\pi(w))$
that a normal monomials $w\in K[X,U]$ is uniquely determined by its leading
monomial $\text{lead}(w)$.

If $\text{deg}_Y(\pi(w)) = 0$ then $\text{deg}_U(w) = 0$ and $\text{lead}(w) = w$. If $\text{deg}_Y(\pi(w)) > 0$ let
\[
\text{lead}(w)=x_1^{a_1}\cdots x_k^{a_k}y_k^{b_k}\cdots x_n^{a_n}y_n^{b_n}
\]
where $k$ is the smallest number for which $b_k \neq 0$. Hence $w$ contains $u_{ik}$ for some $i < k$.
Choose the largest number $j$ for which $w$ contains $u_{jk}$. Then $j < k$ and must be the largest among
these numbers with $a_j \neq 0$.

Indeed, $w$ is a normal monomial. Assume that exists an $l$ for which $a_l \neq 0$ and $j < l < k$.
Then $w$ cannot contain $x_l$ since $x_l$ is covered by $u_{jk}$. But it also cannot contain $u_{lm}$
since $m \geq k$ by the choice of $k$  and if $m = k$ then $w$ contains $u_{lk}$ contrary to the choice of $j$;
so $m > k$ and $u_{jk}$ and $u_{lm}$ would intersect each other.

Therefore $w = u_{jk}w_1$ where $u_{jk}$ is uniquely determined by $w$.
By induction $\text{lead}(w_1)$ determines $w_1$ and so $w$ is determined by $\text{lead}(w)$.
This, of course, implies that $\text{lead}(w)$ are different for different normal monomials
and hence the images $\pi(w)$ of these monomials are linearly independent.

\end{proof}

\section*{Acknowledgements}

The first author is grateful to the Department of Mathematics of
the Wayne State University in Detroit
for the warm hospitality during his
visit when most of this work was carried out.

\end{document}